%% file: linear_image_gp.tex
\documentclass[a4paper, twoside]{scrartcl}
\usepackage[utf8]{inputenc}
\input{./doc-packages.tex}

\input{./doc-newcommands.tex}

\usepackage{etex}
\usepackage[a4paper, top=88pt, bottom=88pt, left=72pt, right=72pt, headsep=16pt, footskip=28pt]{geometry}
\usepackage{hyperref}
\hypersetup{
	breaklinks=true,
	colorlinks=true,
	linkcolor=blue,
	citecolor=blue,
	urlcolor=blue,
}
\newcommand*{\arXiv}[1]{\bgroup\color{blue}\href{https://arxiv.org/abs/#1}{arXiv:#1}\egroup}
\newcommand*{\doi}[1]{\bgroup\color{blue}\href{https://doi.org/#1}{doi:#1}\egroup}
\newcommand*{\email}[1]{\bgroup\color{blue}\href{mailto:#1}{#1}\egroup}
\renewcommand*{\url}[1]{\bgroup\color{blue}\href{#1}{#1}\egroup}
\usepackage{enumitem, moreenum}
\setlist[enumerate]{nosep}
\setlist[itemize]{nosep}
\usepackage{mleftright} \mleftright

\usepackage{xpatch}
\newcommand{\proofheadfont}{\bfseries\sffamily}
\xpatchcmd{\proof}{\itshape}{\proofheadfont}{}{}

\usepackage[labelfont={sf,bf}]{caption}
\usepackage{scrlayer-scrpage, xhfill}
\automark[section]{section}
\setkomafont{pageheadfoot}{\normalcolor\sffamily}
\setkomafont{pagenumber}{\normalfont\normalsize\sffamily}
\clearpairofpagestyles
\let\oldtitle\title
\renewcommand{\title}[1]{\oldtitle{#1}\newcommand{\theshorttitle}{#1}}
\newcommand{\shorttitle}[1]{\renewcommand{\theshorttitle}{#1}}
\let\oldauthor\author
\renewcommand{\author}[1]{\oldauthor{#1}\newcommand{\theshortauthor}{#1}}
\newcommand{\shortauthor}[1]{\renewcommand{\theshortauthor}{#1}}
\cohead{\xrfill[0.525ex]{0.6pt}~\theshorttitle~\xrfill[0.525ex]{0.6pt}}
\cehead{\xrfill[0.525ex]{0.6pt}~\theshortauthor~\xrfill[0.525ex]{0.6pt}}
\newcommand{\theabstract}[1]{\par\bgroup\noindent\textbf{\textsf{Abstract.}} #1\egroup}
\newcommand{\thekeywords}[1]{\par\smallskip\bgroup\noindent\textbf{\textsf{Keywords.}}\newcommand{\and}{ $\bullet$ } #1\egroup}
\newcommand{\themsc}[1]{\par\smallskip\bgroup\noindent\textbf{\textsf{2020 Mathematics Subject Classification.}}\newcommand{\and}{ $\bullet$ } #1\egroup}

\newcommand*{\affilref}[1]{\ref{affiliation#1}}
\newcommand*{\affiliation}[3]{
	\footnotetext[#1]{\label{affiliation#2}#3}
}

\title{Images of Gaussian and other stochastic processes under closed, densely-defined, unbounded linear operators}
\shorttitle{Images of stochastic processes under unbounded linear operators}
\author{
	Tadashi~Matsumoto\textsuperscript{\affilref{Warwick}}
	\and
	T.~J.~Sullivan\textsuperscript{\affilref{Warwick},\affilref{Turing}}
}
\shortauthor{T.~Matsumoto and T.~J.~Sullivan}
\date{\today}

\begin{document}
\maketitle
\affiliation{1}{Warwick}{Mathematics Institute and School of Engineering, University of Warwick, Coventry, CV4 7AL, United Kingdom\newline (\email{t.matsumoto@warwick.ac.uk}, \email{t.j.sullivan@warwick.ac.uk})}
\affiliation{2}{Turing}{Alan Turing Institute, 96 Euston Road, London, NW1 2DB, United Kingdom}

\begin{abstract}\small
	\theabstract{\input{./chunk-abstract.tex}}
	\thekeywords{\input{./chunk-keywords}}
	\themsc{\input{./chunk-msc.tex}}
\end{abstract}

\section{Introduction}
\label{sec:introduction}

It is a basic fact in multivariate probability and statistics that, if $u \sim \Normal(m, C)$ is a normally-distributed $\Reals^{p}$-valued random variable with mean vector $m \in \Reals^{p}$ and covariance matrix $C \in \Reals^{p \times p}$, and $T \in \Reals^{q \times p}$, then the $\Reals^{q}$-valued random variable $T u$ satisfies
\begin{equation}
	\label{eq:image_normal}
	T u \sim \Normal(T m, T C T^{\ast}) ,
\end{equation}
where $T^{\ast} \in \Reals^{p \times q}$ denotes the transpose/adjoint of $T$.
There is a similar formula for the case of a Gaussian stochastic process (GP) $u \sim \GP(m, k)$ on an index set $X$ with mean function $m(x) \defeq \E[u(x)]$ and covariance kernel $k(x_{1}, x_{2}) \defeq \Cov[ u(x_{1}), u(x_{2}) ]$:
if $T$ is a deterministic linear operator acting on functions in $\Reals^{X}$, i.e.~acting pathwise on sample realisations of $u$, then
\begin{equation}
	\label{eq:image_GP}
	T u \sim \GP(T m, T_{1} T_{2} k) ,
\end{equation}
where $T_{j}$ denotes the application of $T$ to the $j$\textsuperscript{th} argument of a function of two or more variables.
Where we write $T_{1} T_{2} k$, alternative notations such as $T k T^{\ast}$, $T T^{\ast} k$, and $T \overline{T} k$ are also commonly used in the literature.
As a concrete prototypical example, as per \citet[Section~9.4]{RasmussenWilliams2006}, if $u \sim \GP(m, k)$ is a GP on $X = \Reals$, then $\frac{\rd u}{\rd x} \sim \GP \bigl( \frac{\rd m}{\rd x} , \frac{\partial^{2} k}{\partial x_{1} \partial x_{2}} \bigr)$, and indeed this special case can be justified in the sense of mean-square differentiability \citep[Theorem~9A-2]{PapoulisPillai2002}.

However, upon close examination, formula \eqref{eq:image_GP} is not completely trivial.
One elementary question is whether the operators $T_{1}$ and $T_{2}$ commute, or at least whether $T_{1} T_{2} k = T_{2} T_{1} k$;
this issue is arguably obscured rather than solved by notations like $T k T^{\ast}$.
A more subtle question is whether \eqref{eq:image_GP} is valid for arbitrary linear operators $T$, and in particular for \emph{unbounded} operators.
This note addresses both these points.

The unbounded case has present-day practical relevance.
Many recent applications of GPs, such as linearly- and algorithmically-constrained GPs \citep{Jidling2017, LangeHegermann2018, LangeHegermann2021, BesginowLangeHegermann2022}, numerical and physics-informed GPs \citep{PangKarniadakis2020, RaissiEtAl2018} and probabilistic numerical methods for differential equations \citep{CockayneOatesSullivanGirolami2017a, HennigOsborneKersting2022, OatesCockayneAykroydGirolami2019, PfortnerEtAl2022, Saerkkae2011}, seek to apply \eqref{eq:image_GP} in the case that $T$ is an ordinary or partial differential operator.
These applications typically then wish to form the conditional distribution of $u$ given an observation of $T u$ --- for which it is necessary to know that $T u$ is a GP and to know its mean and covariance structure.
However, in the main, these papers mention few requirements on $T$ and none provides a proof of \eqref{eq:image_GP} in the general unbounded case.
Some, such as \citet[Appendix~B]{PfortnerEtAl2022} and \citet[Section~2 and Appendix~A]{HaerkoenenEtAl2023}, do state \eqref{eq:image_GP} as a theorem, but add the restriction that $T$ must be a bounded/continuous operator between Banach/Fr\'echet spaces of functions.
In the case of a differential operator, this can only be achieved by carefully setting up the function spaces in an artificial and inflexible way --- e.g.~by making $T = \frac{\rd}{\rd x}$ a bounded operator from $\mathcal{C}^{1} ([0, 1]; \Reals)$ into $\mathcal{C}^{0} ([0, 1]; \Reals)$ with their usual norms --- that complicates repeated application of $T$ (as needed, e.g., for filtering and data assimilation).
Other authors, e.g.~\citet{LangeHegermann2021}, give a proof for a measurable linear operator $T$ that commutes with expectation, but neither stress the sense in which the expectation is meant nor give sufficient conditions for this commutativity.

This note aims to give a short, self-contained, rigorous proof of \eqref{eq:image_GP}, treating $T$ as an unbounded operator.
Our main result, which rests upon Hille's theorem for the Bochner integral, is the following:

\begin{theorem}
	\label{thm:main}
	Let $X$ be a non-empty set, let $\U, \V \subseteq \Reals^{X}$ be real Banach spaces, and let $T \colon \dom(T) \subseteq \U \to \V$ be a closed operator.
	Let $(\Omega, \Sigma, \P)$ be a probability space, and let $u \colon \Omega \times X \to \Reals$ be a stochastic process such that $u(\omega, \quark) \in \dom(T)$ $\P$-a.s.
	Let $v \colon \Omega \times X \to \Reals$ be given $\P$-a.e.~by $v(\omega, x) \defeq (T u(\omega, \quark)) (x)$.
	\begin{enumerate}[label=(\alph*)]
		\item
		\label{item:main_mean}
		(Mean function.)
		Suppose that $u$ and $v$ have Bochner mean functions $m_{u} \in \U$ and $m_{v} \in \V$, i.e.~$\int_{\Omega} \norm{ u(\omega, \quark) }_{\U} \, \P(\rd \omega)$ and $\int_{\Omega} \norm{ v(\omega, \quark) }_{\V} \, \P(\rd \omega)$ are finite.
		Then $m_{v} = T m_{u}$.

		\item
		\label{item:main_cov}
		(Covariance function.)
		Suppose in addition to \ref{item:main_mean} that $T$ has a densely-defined adjoint $T^{\ast}$ and that $\int_{\Omega} \norm{ u(\omega, \quark) }_{\U}^{2} \, \P(\rd \omega)$ and $\int_{\Omega} \norm{ v(\omega, \quark) }_{\V}^{2} \, \P(\rd \omega)$ are finite.
		Then $u$ and $v$ have covariance functions $k_{u}$ and $k_{v}$ respectively, and $k_{v} = T_{1} T_{2} k_{u} = T_{2} T_{1} k_{u}$, where $T_{1}$ (resp.~$T_{2}$) denotes the densely-defined closure of the tensor product operator $T \otimes I$ (resp.~$I \otimes T$).

		\item
		\label{item:main_GP}
		(Gaussianity.)
		Suppose in addition to \ref{item:main_mean} and \ref{item:main_cov} that the linear span of all point evaluation functionals is weakly-$\ast$ sequentially dense\footnote{That is, each $\ell \in \U'$ must admit a sequence of finite linear combinations of point evaluations that converges weakly-$\ast$ in $\U'$ to $\ell$.
		Some sufficient conditions for this are discussed in \Cref{lem:GP_to_GRV_revised}, but we note here that this criterion is satisfied by any separable reproducing kernel Hilbert space, any separable and quasi-reflexive reproducing kernel Banach space in the sense of \citet{LinZhangZhang2022}, and any separable Banach space of compactly-supported continuous functions (with respect to the supremum norm).} in $\U'$, that \mbox{$u \sim \GP(m_{u}, k_{u})$}, and that \mbox{$\int_{\Omega} \norm{ v(\omega, \quark) }_{\V}^{n} \, \P(\rd \omega)$} is finite for each $n \in \Naturals$.
		Then \mbox{$v \sim \GP(T m_{u}, T_{1} T_{2} k_{u})$}.
	\end{enumerate}
\end{theorem}

We find it likely that \Cref{thm:main} is not wholly new;
on the other hand, we have not found it stated straightforwardly and proved rigorously at this level of generality in the literature.

\begin{remark}[Necessity of assumptions]
	It is necessary for $\U$ and $\V$ to be Banach spaces to even define the Bochner integral, and our appeal to Hille's theorem requires $u$ etc.~to be Bochner integrable and $T$ to be closed.
	Without these assumptions, standard counterexamples to the principle of differentiation under the integral sign can be used to invalidate claim \ref{item:main_mean};
	cf.~\Cref{example:Talvila}.
	In \ref{item:main_cov}, the densely-defined adjoint of $T$ ensures that $T_{i}$ is closed, which is again needed for Hille's theorem.
	Finally, in \ref{item:main_GP}, continuity of point evaluation is needed to ensure that the point evaluations are weakly-$\ast$ dense in $\U'$, which allows $u$ to be seen as a bona fide $\U$-valued Gaussian random variable with finite moments of all orders.
	Finiteness of the moments of $T u$ must be assumed separately from those of $u$ in order to use Hille's theorem, since $T$ is not necessarily a bounded operator.
\end{remark}

The necessary notation and preliminaries are set out in \Cref{sec:setup}, and readers familiar with tensor products, unbounded operators, stochastic processes, and Bochner integration may skip this section.
With this basis, \Cref{sec:proof_of_main} gives the proof of \Cref{thm:main}.
\Cref{sec:PDE} gives an extended example, applying \Cref{thm:main} to an elliptic partial differential operator.

\section{Setup and notation}
\label{sec:setup}

Let $(\Omega, \Sigma, \P)$ be a probability space that is rich enough to serve as a common domain of definition for all random variables and stochastic processes under consideration.
Calligraphic letters $\U$, $\V$, etc.~will denote real Banach spaces, and $X$ will denote a non-empty index set.
The topological dual space $\U'$ consists of all bounded/continuous linear functionals from $\U$ to $\Reals$, and we write $\dualprod{ \ell }{ f } \in \Reals$ for the action of $\ell \in \U'$ upon $f \in \U$.

\subsection{Tensor product functions and tensor product spaces}

Given functions $u_{1}, u_{2} \colon X \to \Reals$, define $u_{1} \otimes u_{2} \colon X^{2} \to \Reals$ by $(u_{1} \otimes u_{2}) (x_{1}, x_{2}) \defeq u_{1} (x_{1}) u_{2} (x_{2})$.
When $\U, \V \subseteq \Reals^{X}$ are Banach spaces, we write $\U \otimes \V \defeq \SPAN \set{ u \otimes v }{ u \in \U , v \in \V }$ and $\U \hatotimes \V$ for the Banach completion of $\U \otimes \V$ with respect to some norm satisfying $\norm{ u \otimes v } \leq C \norm{ u }_{\U} \norm{ v }_{\V}$ for some constant $C > 0$ and all $u \in \U$, $v \in \V$;
this may be, but does not have to be, the projective tensor product $\U \hatotimes_{\pi} \V$ of $\U$ and $\V$ \citep[Chapter~2]{Ryan2002}.

\subsection{Stochastic processes}

Given a stochastic process $u \colon \Omega \times X \to \Reals$, i.e.~a random function $u \colon \Omega \to \Reals^{X}$, its \defterm{mean function} $m_{u} \colon X \to \Reals$ is defined pointwise by
\begin{equation}
	\label{eq:mean_function}
	m_{u}(x) \defeq \E [ u(x) ] = \int_{\Omega} u(\omega, x) \, \P (\rd \omega)
\end{equation}
and its \defterm{covariance function} $k_{u} \colon X \times X \to \Reals$ is defined pointwise by
\begin{align*}
	k_{u}(x_{1}, x_{2}) \defeq \Cov [ u(x_{1}), u(x_{2}) ] \defeq \E \left[ \vphantom{\big|} (u(x_{1}) - m_{u}(x_{1})) (u(x_{2}) - m_{u}(x_{2})) \right] .
\end{align*}
The process $u$ is called a \defterm{Gaussian process} (GP) with mean function $m_{u}$ and covariance function $k_{u}$, written $u \sim \GP(m_{u}, k_{u})$, if, for every finite set $\{ x_{1}, \dots, x_{J} \} \subseteq X$,
\begin{align*}
	\begin{pmatrix} u(x_{1}) \\ \vdots \\ u(x_{J}) \end{pmatrix}
	\sim
	\Normal \left(
		\begin{pmatrix} m_{u}(x_{1}) \\ \vdots \\ m_{u}(x_{J}) \end{pmatrix}
		,
		\begin{pmatrix}
			k_{u}(x_{1}, x_{1}) & \cdots & k_{u}(x_{1}, x_{J}) \\
			\vdots & \ddots & \vdots \\
			k_{u}(x_{J}, x_{1}) & \cdots & k_{u}(x_{J}, x_{J})
		\end{pmatrix}
	\right) .
\end{align*}

GPs are closely related to Gaussian measures and Gaussian random variables \citep[Section~2.3]{Bogachev1998}, and the following lemma, proved in \Cref{app:technical}, is needed for \Cref{thm:main}\ref{item:main_GP}.
We emphasise that the conditions given in \Cref{lem:GP_to_GRV_revised}\ref{item:GP_to_GRV_revised_2} are sufficient but not necessary for the conclusion of \Cref{lem:GP_to_GRV_revised}\ref{item:GP_to_GRV_revised_3}, and similar results have recently been provided by \citet[Theorem~B.6 et seq.]{PfortnerEtAl2022}.

\begin{lemma}
	\label{lem:GP_to_GRV_revised}
	Let $\U \subseteq \Reals^{X}$ be a Banach space.
	Let $\delta_{x}$ denote the point evaluation functional at $x \in X$, i.e.~$\dualprod{ \delta_{x} }{ f } \defeq f(x)$ for $f \in \U$, and suppose that $\delta_{x} \in \U'$ for each $x \in X$.
	Let $\Delta \defeq \SPAN \{ \delta_{x} \}_{ x \in X }$ and let $\Delta_{(1)}$ denote the set of all weak-$\ast$ limits of sequences in $\Delta$.
	\begin{enumerate}[label=(\alph*)]
		\item
		\label{item:GP_to_GRV_revised_1}
		If $\U$ is separable, then the only weakly-$\ast$ sequentially closed subspace of $\U'$ that contains $\Delta$ is $\U'$.
		\item
		\label{item:GP_to_GRV_revised_2}
		If $\U$ is separable and quasi-reflexive (i.e.\ $\U$ has finite codimension in $\U''$ under the canonical embedding) or is a separable Banach space of bounded functions under the supremum norm, such that each $u \in \U$ attains its supremum norm, then  $\U' = \Delta_{(1)}$.
		\item
		\label{item:GP_to_GRV_revised_3}
		Let $u \sim \GP(m_{u}, k_{u})$ be a GP over $X$ with sample paths a.s.~in $\U$.
		If $\U' = \Delta_{(1)}$, then $u$ is a weakly Gaussian random variable in $\U$ in the sense that $\dualprod{ \ell }{ u }$ is Gaussian in $\Reals$ for every $\ell \in \U'$, and is strongly Gaussian (i.e.\ its law is a Gaussian measure on the Borel $\sigma$-algebra) if $u$ is essentially separably valued.
	\end{enumerate}
\end{lemma}

\subsection{Closed operators}
\label{sec:closed_operators}

Thorough treatments of unbounded operators are offered by e.g.~\citet[Section~III.5]{Kato1995} and \citet[Chapter~VIII]{ReedSimon1980}, the latter considering only the Hilbert space case.

A linear operator $T \colon \dom(T) \subseteq \U \to \V$, defined on a linear subspace $\dom(T)$ of $\U$, is a \defterm{closed operator} if, whenever $(u_{n})_{n \in \Naturals} \subset \dom(T)$ is such that $(u_{n}, T u_{n})$ converges to some $(u, v)$, it follows that $u \in \dom(T)$ and $v = T u$;
equivalently, the graph $\graph(T) \defeq \set{ (u, T u) }{ u \in \dom(T) }$ of $T$ is a closed subspace of the direct sum $\U \oplus \V$ with its norm $\norm{ (u, v) }_{\U \oplus \V} \defeq \norm{ u }_{\U} + \norm{ v }_{\V}$.
An operator $T$ is called \defterm{closable} if it can be extended to a closed operator $\overline{T}$, and $\graph(\overline{T})$ is the closure of $\graph(T)$ in $\U \oplus \V$.
Note well that closedness of $T$ is unrelated to closedness of $\dom(T)$.

Any operator $T$ induces a norm on $\dom(T)$ via $\norm{ u }_{T} \defeq \norm{ u }_{\U} + \norm{ T u }_{\V}$, and it is easy to verify that $T$ is a bounded operator with at most unit operator norm from $(\dom(T), \norm{ \quark }_{T})$ into $(\V, \norm{ \quark }_{\V})$.
Whenever $\U$ and $\V$ are Banach and $T$ is a closed operator, $\graph(T)$ is a Banach subspace of $\U \oplus \V$ and $(\dom(T), \norm{ \quark }_{T})$ is also a Banach space.

This note's prototypical example is in fact the standard example of a closed operator defined on a dense but proper subspace, and which cannot be extended to a bounded operator on the whole space:
$\frac{\rd}{\rd x} \colon \dom(\frac{\rd}{\rd x}) \defeq \mathcal{C}^{1}([0, 1]; \Reals) \subsetneq \mathcal{C}^{0}([0, 1]; \Reals) \to \mathcal{C}^{0}([0, 1]; \Reals)$.
A more involved example, involving a partial differential operator and the solution of an elliptic partial differential equation (PDE), is discussed in \Cref{sec:PDE}.

\subsection{Bochner integration}
\label{sec:Bochner_integration}

We now recall the essential elements of Bochner integration of Banach-space-valued random variables;
for a comprehensive treatment, see e.g.~\citet[Section~II.2]{DiestelUhl1977}.

For $E \in \Sigma$, $\one_{E} \colon \Omega \to \Reals$ denotes the \defterm{indicator function} of $E$ given by
\[
	\one_{E}(\omega) \defeq
	\begin{cases}
		1, & \text{if $\omega \in E$,} \\
		0, & \text{if $\omega \notin E$.}
	\end{cases}
\]
A \defterm{simple} $\U$-valued random variable $s \colon \Omega \to \U$ takes the form $s(\omega) = \sum_{j = 1}^{J} u_{j} \one_{E_{j}} (\omega)$ for some $J \in \Naturals$, $u_{j} \in \U$, and $E_{j} \in \Sigma$;
its \defterm{expected value} or \defterm{integral} is
\begin{align*}
	\E [s] \equiv \int_{\Omega} s(\omega) \, \P(\rd \omega) \defeq \sum_{j = 1}^{J} u_{j} \P(E_{j}) \in \U .
\end{align*}

A random variable $u \colon \Omega \to \U$ is called \defterm{strongly measurable} if there exist simple random variables $s_{n}$ such that $\lim_{n \to \infty} s_{n}(\omega) = u(\omega)$ for $\P$-a.e.~$\omega \in \Omega$.
If also $\E [ \norm{ u - s_{n} }_{\U} ] \to 0$ as $n \to \infty$, then $u$ is called \defterm{Bochner integrable} and its \defterm{Bochner integral} is defined by
\begin{align*}
	\E [ u ] \equiv \int_{\Omega} u(\omega) \, \P (\rd \omega) \defeq \lim_{n \to \infty} \E [ s_{n} ] \in \U .
\end{align*}
Equivalently, $u$ is Bochner integrable if and only if it is strongly measurable and $\E [ \norm{ u }_{\U} ]$ is finite.
If $\U$ happens to be a separable space, or if $u$ is essentially separably valued, then Pettis' measurability theorem \citep[Theorem~II.1.2]{DiestelUhl1977} ensures that strong measurability coincides with Borel measurability and also with \defterm{weak measurability}, i.e.~measurability of every $\dualprod{ \ell }{ u } \colon \Omega \to \Reals$ for $\ell \in \U'$, which is often easier to verify in practice.

If $u \colon \Omega \to \U$ is Bochner integrable and $T \colon \U \to \V$ is a bounded operator with values in another Banach space $\V$, then it is easily verified\footnote{The key insight is that, if the operator norm $\norm{ T }_{\textup{op}}$ is finite, then $\E [ \norm{ v }_{\V} ] \leq \norm{ T }_{\textup{op}} \E [ \norm{ u }_{\U} ]$.} that $v \defeq T u$ is Bochner integrable and $\E [ v ] = T \E [ u ]$.
Our interest, however, lies in unbounded operators, for which the following result will be essential:

\begin{theorem}[Hille's theorem; {\citealp[Theorem~II.2.6]{DiestelUhl1977}}]
	\label{thm:Hille}
	Let $\U$ and $\V$ be real Banach spaces and let $T \colon \dom(T) \subseteq \U \to \V$ be a closed operator.
	If $u \colon \Omega \to \U$ and $T u \colon \Omega \to \V$ are both Bochner integrable, then $\E [ T u ] = T \E [ u ] \in \V$.
\end{theorem}

Note that, when $\U \subseteq \Reals^{X}$, $u$ need not be Bochner integrable in order to define its mean function pointwise on $X$ as is done in \eqref{eq:mean_function}.
However, if $m_{u}$ is the mean of $u$ in merely this weak sense, then one will have no recourse to Hille's theorem, which is essential for our proof of \Cref{thm:main}.

\begin{remark}[TANSTAAFL!]
	\label{rmk:TANSTAAFL}
	The hypotheses of Hille's theorem are equivalent to asking that $u$ be Bochner integrable with respect to $\norm{ \quark }_{T}$ and that $(\dom(T), \norm{ \quark }_{T})$ be a Banach space, on which $T$ is automatically a bounded $\V$-valued operator.
	However, there is ``no free lunch'':
	either one works with the weaker norm $\norm{ \quark }_{\U}$ at the cost of working with a closed, unbounded operator;
	or one works with a $(\norm{ \quark }_{T}, \norm{ \quark }_{\V})$-bounded operator at the cost of needing $\norm{ \quark }_{T}$ to be a Banach norm, which holds only when $T$ is closed anyway.
\end{remark}

\section{Proof of \texorpdfstring{\Cref{thm:main}}{Theorem \ref{thm:main}}}
\label{sec:proof_of_main}

We first establish claim \ref{item:main_mean} regarding the mean function.

\begin{proof}[Proof of \Cref{thm:main}\ref{item:main_mean}]
	Since the means $m_{u}$ and $m_{v}$ exist in the Bochner sense and the operator $T$ is closed,
	\begin{align*}
		m_{v}
		= \E [ v ]
		= \E [ T u ]
		= T \E [ u ]
		= T m_{u} ,
	\end{align*}
	as claimed, where the equality $\E [ T u ] = T \E [ u ]$ follows from Hille's theorem.
\end{proof}

Hille's theorem can be seen as a generalisation of the principle of differentiation under the integral sign, and so the standard counterexamples for differentiation under the integral sign can be used to show that \Cref{thm:main}\ref{item:main_mean} fails when, for example, $T u$ is not Bochner integrable:

\begin{example}[\citealp{Talvila2001a, Talvila2001b}]
	\label{example:Talvila}
	Let $X \defeq \Reals$, $\Omega \defeq (0, \infty)$, and let $\P$ be the exponential distribution on $\Omega$ with parameter $1$.
	Consider the process $u \colon \Omega \times X \to \Reals$ given by
	\[
		u(\omega, x) \defeq \cos ( \omega x ) \sin ( \omega^{2} ) \exp ( \omega ) .
	\]
	Then $u$ has the mean function $m_{u} \colon X \to \Reals$ given by
	\[
		m_{u}(x) = \int_{0}^{\infty} \cos ( \omega x ) \sin ( \omega^{2} ) \, \rd \omega = \frac{1}{2} \sqrt{ \frac{\pi}{2} } \left( \cos \frac{x^{2}}{4} - \sin \frac{x^{2}}{4} \right) .
	\]
	Consider the image $v$ of $u$ under $T \defeq \frac{\rd}{\rd x}$, i.e.~$v(\omega, x) = - \omega \sin ( \omega x ) \sin ( \omega^{2} ) \exp ( \omega )$.
	The process $v$ has no mean function, since $\int_{0}^{\infty} \omega \sin ( \omega x ) \sin ( \omega^{2} ) \, \rd \omega$ diverges for every $x \in X$;
	on the other hand, $m_{u}$ is differentiable with
	\[
		\frac{\rd}{\rd x} m_{u}(x) = - \frac{x}{4} \sqrt{ \frac{\pi}{2} } \left( \sin \frac{x^{2}}{4} + \cos \frac{x^{2}}{4} \right) .
	\]
\end{example}

Before proving claim \ref{item:main_cov} regarding the covariance kernel, we carefully define the following versions $T_{j}$ of $T$ that act on one argument of a function of two arguments, in the sense that $T_{j}$ plays the role of $\frac{\partial}{\partial x_{j}}$ when $T = \frac{\rd}{\rd x}$ on $\dom(T) = \mathcal{C}^{1}([0, 1]; \Reals)$.

Consider the tensor product linear operator $T \otimes I$, which sends each elementary product function $u_{1} \otimes u_{2}$ to $(T u_{1}) \otimes u_{2}$.
By assumption, $T$ has a densely-defined adjoint $T^{\ast}$, and so $T^{\ast} \otimes I$ is a densely-defined adjoint for $T \otimes I$, which implies that $T \otimes I$ is closable \citep[Theorem~III.5.28]{Kato1995}.
Thus, let $T_{1} \defeq \overline{ T \otimes I }$ be its closure, which is densely defined in $\U \hatotimes \U$ and takes values in $\V \hatotimes \U$.
As a small abuse of notation, we also write $T_{1}$ for the operator defined in the same way within $\U \hatotimes \V$ and taking values in $\V \hatotimes \V$.
The closed and densely-defined operator $T_{2} \defeq \overline{I \otimes T}$ is defined similarly.

One sufficient condition for $T$ to have a densely-defined adjoint is for it to be closed\footnote{Or even closable, but closedness of $T$ was already assumed for \Cref{thm:main}\ref{item:main_mean}.} and densely defined in a reflexive space $\U$ \citep[Theorem~III.5.29]{Kato1995}.
In other non-reflexive cases of interest, the existence of a densely-defined adjoint $T^{\ast}$ can be verified directly.
For a simple example, consider $\U \defeq \mathcal{C}_{0}^{0}([0, 1]; \Reals)$ and $T \defeq \frac{\rd}{\rd x}$ defined on $\dom(T) \defeq \mathcal{C}_{0}^{1}([0, 1]; \Reals)$.
In this case, the topological dual space $\U'$ is the space $\mathfrak{M}([0, 1])$ of finite-variation signed Radon measures on $[0, 1]$, equipped with the total variation norm.
A quick integration by parts, simplified by the zero boundary conditions, shows that $T^{\ast} \mu = - \rho' \, \rd x$ if $\mu \in \mathfrak{M}([0, 1])$ has differentiable Lebesgue density $\rho$, the collection of such $\mu$ being dense in $\mathfrak{M}([0, 1])$.

We are now in a position to establish claim \ref{item:main_cov}.

\begin{proof}[Proof of \Cref{thm:main}\ref{item:main_cov}]
	The square-integrability assumptions on $u$ and $T u$ and the assumption that $\norm{ u \otimes v } \leq C \norm{ u }_{\U} \norm{ v }_{\V}$ together imply that $T_{j} u \otimes u$ is also Bochner integrable;
	the same is true after centering the processes by subtracting their means.
	Thus, $u$ and $v$ have covariance functions $k_{u}$ and $k_{v}$ that are well-defined as Bochner integrals.
	The discussion above has already established the closedness of $T_{1}$ and $T_{2}$.
	Thus, the hypotheses needed for the two applications of Hille's theorem below are satisfied, and we have
	\begin{align*}
		k_{v}
		& = \E \left[ \vphantom{\big|} ( v - m_{v} ) \otimes ( v - m_{v} ) \right] \\
		& = \E \left[ \vphantom{\big|} ( T u - T m_{u} ) \otimes ( T u - T m_{u} ) \right] \\
		& = \E \left[ \vphantom{\big|} T_{1} \bigl( ( u - m_{u} ) \otimes ( T u - T m_{u} ) \bigr) \right] \\
		& = T_{1} \E \left[ \vphantom{\big|} ( u - m_{u} ) \otimes ( T u - T m_{u} ) \right] & & \text{(Hille's theorem for $T_{1}$)} \\
		& = T_{1} \E \left[ \vphantom{\big|} T_{2} \bigl( ( u - m_{u} ) \otimes ( u - m_{u} ) \bigr) \right] \\
		& = T_{1} T_{2} \E \left[ \vphantom{\big|} ( u - m_{u} ) \otimes ( u - m_{u} ) \right] & & \text{(Hille's theorem for $T_{2}$)} \\
		& = T_{1} T_{2} k_{u} ,
	\end{align*}
	as claimed.
	Finally, note that we could equally well have applied Hille's theorem first to $T_{2}$ and then to $T_{1}$, thus showing that $k_{v} = T_{2} T_{1} k_{u}$.
\end{proof}

Finally, we treat claim \ref{item:main_GP} regarding Gaussianity of $v$.
The following argument using cumulants echoes that of \citet{LangeHegermann2021}, but is more careful in its appeal to Hille's theorem to justify the interchange of $T$-like operators and Bochner expectations.

For $n \in \Naturals$, define the \defterm{$n$\textsuperscript{th} cumulant function} $\kappa_{u}^{(n)} \colon X^{n} \to \Reals$ by
\begin{equation}
	\label{eq:cumulant_pointwise}
	\kappa_{u}^{(n)} ( x_{1}, \dots, x_{n} ) \defeq \sum_{P \in \partitions(n)} (-1)^{\absval{ P } - 1} ( \absval{ P } - 1 )! \prod_{S \in P} \E \left[ \prod_{i \in S} u(x_{i}) \right] ,
\end{equation}
where $\partitions(n)$ denotes the set of all partitions of $\{ 1, \dots, n \}$.
We will use the fact that a random vector is normally distributed if and only if its cumulants of order $n \geq 3$ vanish \citep{Marcinkiewicz1939}, and hence that a stochastic process $u$ on $X$ is a GP if and only if, for $n \geq 3$, $\kappa_{u}^{(n)}$ is the zero function on $X^{n}$.
Furthermore, in order to relate $\kappa_{u}^{(n)}$ and $\kappa_{T u}^{(n)}$, we will interpret \eqref{eq:cumulant_pointwise} not pointwise but as a Bochner expectation of $\U \hatotimes \cdots \hatotimes \U$-valued random variables, i.e.
\begin{align*}
	\kappa_{u}^{(n)} \defeq \sum_{P \in \partitions(n)} (-1)^{\absval{ P } - 1} ( \absval{ P } - 1 )! \prod_{S \in P} \E \left[ u^{\otimes S} \right] ,
\end{align*}
where, for $S \subseteq \{ 1, \dots, n \}$, $u^{\otimes S} \colon X^{n} \to \Reals$ is the function
\[
	u^{\otimes S} (x_{1}, \dots, x_{n}) \defeq \prod_{i = 1}^{n} z_{i}
	\quad
	\text{with}
	\quad
	z_{i} \defeq
	\begin{cases}
		u(x_{i}), & \text{if $i \in S$,} \\
		1, & \text{if $i \notin S$.}
	\end{cases}
\]
Furthermore, employing the same closability arguments as before, we write
\[
	T_{S} \defeq \overline{ \bigotimes_{i = 1}^{n} L_{i} }
	\quad
	\text{with}
	\quad
	L_{i} \defeq
	\begin{cases}
		T_{i}, & \text{if $i \in S$,} \\
		I, & \text{if $i \notin S$.}
	\end{cases}
\]
For example, in our prototypical setting, $T_{\{ 2, 3 \}}$ plays the role of $\frac{\partial^{2}}{\partial x_{2} \partial x_{3}}$.

\begin{lemma}
	\label{lem:cumulants_of_u_and_Tu}
	Let $T \colon \dom(T) \subset \U \to \V$ be closed and densely defined, with a densely-defined adjoint $T^{\ast}$.
	Let $n \in \Naturals$ be such that both $\E [ \norm{ u }_{\U}^{n} ]$ and $\E [ \norm{ T u }_{\V}^{n} ]$ are finite.
	Then
	\begin{equation}
		\label{eq:cumulants_of_u_and_Tu}
		\kappa_{T u}^{(n)} = T_{\{ 1, \dots, n \}} \kappa_{u}^{(n)} .
	\end{equation}
\end{lemma}

\begin{proof}
	Note that the assumptions on the moments of $u$ and $T u$ are necessary to ensure the Bochner integrability of the tensor product functions and applications of Hille's theorem appearing below;
	if $T$ were a bounded operator, then finiteness of $\E [ \norm{ u }_{\U}^{n} ]$ would imply finiteness of $\E [ \norm{ T u }_{\V}^{n} ]$.

	Calculating directly, we have
	\begin{align*}
		\kappa_{T u}^{(n)}
		& = \sum_{P \in \partitions(n)} (-1)^{\absval{ P } - 1} ( \absval{ P } - 1 )! \prod_{S \in P} \E \left[ (T u)^{\otimes S} \right] && \\
		& = \sum_{P \in \partitions(n)} (-1)^{\absval{ P } - 1} ( \absval{ P } - 1 )! \prod_{S \in P} \E \left[ T_{S} u^{\otimes S} \right] && \textup{(definition of $T_{S}$)} \\
		& = \sum_{P \in \partitions(n)} (-1)^{\absval{ P } - 1} ( \absval{ P } - 1 )! \prod_{S \in P} T_{S} \E \left[ u^{\otimes S}  \right] && \textup{(Hille's theorem)} \\
		& = \sum_{P \in \partitions(n)} (-1)^{\absval{ P } - 1} ( \absval{ P } - 1 )! \, T_{\{ 1, \dots, n \}} \prod_{S \in P} \E \left[ u^{\otimes S}  \right] && \textup{(since $\textstyle \biguplus_{S \in P} S = \{ 1, \dots, n \}$)} \\
		& = T_{\{ 1, \dots, n \}} \kappa_{u}^{(n)} && \textup{(linearity of $T_{\{ 1, \dots, n \}}$)}
	\end{align*}
	and this establishes \eqref{eq:cumulants_of_u_and_Tu}.
	Note that the appeal to Hille's theorem in the third line can be seen either as $\absval{ S }$ applications of Hille's theorem, one for each $T_{\{ i \}}$, $i \in S$, or as one application of Hille's theorem for $T_{S}$.
	Note also that these operators may be taken outside the expectation in any order while yielding expressions that are equal to $\kappa_{T u}^{(n)}$, and thus equal to one another.
\end{proof}

\begin{proof}[Proof of \Cref{thm:main}\ref{item:main_GP}]
	By \Cref{lem:GP_to_GRV_revised}, $u$ is a (strongly) Gaussian random variable in the separable Banach space $\U$.
	Hence, by Fernique's theorem \citep{Fernique1970}, it has finite moments of all orders.
	Furthermore, since $u$ is Gaussian, it follows that $\kappa_{u}^{(n)} \equiv 0$ for $n \geq 3$.

	By assumption, $v \defeq T u$ also has finite moments of all orders, and \Cref{lem:cumulants_of_u_and_Tu} implies that $\kappa_{v}^{(n)} \equiv 0$ for $n \geq 3$, which shows that $v$ is a GP, as claimed, and its mean and covariance functions were determined by \Cref{thm:main}\ref{item:main_mean}--\ref{item:main_cov}.
\end{proof}

\section{Application to GP-based solution of PDEs}
\label{sec:PDE}

As noted in \Cref{sec:introduction,sec:closed_operators}, the standard example of the framework discussed in this note is the derivative operator on a space of continuous paths on an interval in $\Reals$.
We now give a slightly more involved worked example, one that is representative of the application of the above results to the numerical solution of an elliptic PDE using a GP model of the kind common in spatial statistics and machine learning.
This \emph{probabilistic meshless method} (PMM) \citep{CockayneOatesSullivanGirolami2017a,OatesCockayneAykroydGirolami2019} can be seen as a form of generalised kriging \citep{Krige1951,Matheron1963} with derivative-based observations, and the posterior mean function \eqref{eq:PMM_post_mean} of the PMM coincides with the symmetric collocation method \citep{CialencoEtAl2012,Fasshauer1999}.
The additional benefit of the GP representation is that the posterior covariance structure \eqref{eq:PMM_post_cov} gives a meaningful estimate of the discretisation uncertainty vis-\`a-vis the PDE solution \citep{HennigOsborneKersting2022}.

Let $X \subset \Reals^{d}$ be a bounded, open domain with smooth boundary $\partial X$ and compact closure $\overline{X}$.
Let $L$ be a second-order linear differential operator of the form
\begin{align*}
	L u \defeq - \nabla \cdot ( a \nabla u ) + b \cdot \nabla u - \nabla \cdot ( b u ) + c u
\end{align*}
with $\mathcal{C}^{\infty}$ coefficient functions $a \colon \overline{X} \to \Reals^{d \times d}$, $b \colon \overline{X} \to \Reals^{d}$, and $c \colon \overline{X} \to \Reals$.
Assume that $a(x)$ is symmetric for each $x \in X$, so that $L$ is formally self-adjoint with respect to the usual $L^{2}$ inner product, and assume also that the eigenvalues of $a(x)$ are bounded below, uniformly in $x$, by $\lambda_{\min} > 0$, so that $L$ is uniformly elliptic.

In the PMM, we are interested in forming a posterior GP whose sample paths are approximate solutions to the PDE
\begin{equation}
	\label{eq:elliptic_pde_1}
	\begin{cases}
		L u = f & \text{in $X$,} \\
		B u = g & \text{on $\partial X$,}
	\end{cases}
\end{equation}
for prescribed $f \colon X \to \Reals$ and $g \colon \partial X \to \Reals$, where $B \colon u \mapsto u|_{\partial X}$ denotes the Dirichlet boundary value / Sobolev trace operator.
The sense in which the GP paths will be approximate solutions, in the spirit of collocation methods \citep{CialencoEtAl2012,Fasshauer1999}, is that the GP will be conditioned so that \eqref{eq:elliptic_pde_1} holds on a finite set of points in $\overline{X}$.

For the next paragraph we reason somewhat informally.
Fix finite point sets $\bs{\xi}^{\circ} \defeq \{ \xi_{i}^{\circ} \}_{i = 1}^{I} \subset X$ and $\bs{\xi}^{\partial} \defeq \{ \xi_{j}^{\partial} \}_{j = 1}^{J} \subset \partial X$, let $\bs{\xi} \defeq \bs{\xi}^{\circ} \cup \bs{\xi}^{\partial}$, and let
\begin{align*}
	\bs{y} & \defeq
	\begin{bmatrix}
		\vphantom{\Big|}
		f(\xi_{1}^{\circ}) &
		\cdots &
		f(\xi_{I}^{\circ}) &
		g(\xi_{1}^{\partial}) &
		\cdots &
		g(\xi_{J}^{\partial})
	\end{bmatrix}^{\intercal}
	 \in \Reals^{(I + J) \times 1} .
\end{align*}
After positing a centred GP prior $u \sim \GP(0, k)$, the PMM conditions this GP on the linear observation $(L u (\bs{\xi}^{\circ}), B u (\bs{\xi}^{\partial}) ) = \bs{y}$;
the usual conditioning procedure for GPs yields a the PMM solution $\widetilde{u} \sim \GP(\widetilde{m}, \widetilde{k})$ to the PDE \eqref{eq:elliptic_pde_1}, with posterior mean $\widetilde{m} \colon X \to \Reals$ and covariance $\widetilde{k} \colon X \times X \to \Reals$ for $v$ given by
\begin{align}
	\label{eq:PMM_post_mean}
	\widetilde{m}(x) & \defeq \bigl( T_{2} k(x, \bs{\xi}) \bigr) \bigl( T_{1} T_{2} k (\bs{\xi}, \bs{\xi}) \bigr)^{-1} \bs{y} , \\
	\label{eq:PMM_post_cov}
	\widetilde{k}(x, x') & \defeq k(x, x') - \bigl( T_{2} k(x, \bs{\xi}) \bigr) \bigl( T_{1} T_{2} k (\bs{\xi}, \bs{\xi}) \bigr)^{-1} \bigl( T_{1} k(\bs{\xi}, x') \bigr) ,
\end{align}
where $T_{1} k(\bs{\xi}, x) \in \Reals^{(I + J) \times 1}$, $T_{2} k(x, \bs{\xi}) \in \Reals^{1 \times (I + J)}$, and $T_{1} T_{2} k(\bs{\xi}, \bs{\xi}) \in \Reals^{(I + J) \times (I + J)}$ are defined by
\begin{align*}
	T_{1} k(\bs{\xi}, x)
	& \defeq
	\begin{bmatrix}
		\vphantom{\Big|}
		L_{1} k(\xi_{1}^{\circ}, x) &
		\cdots &
		L_{1} k(\xi_{I}^{\circ}, x) &
		k(\xi_{1}^{\partial}, x) &
		\cdots &
		k(\xi_{J}^{\partial}, x)
	\end{bmatrix}^{\intercal} , \\
	T_{2} k(x, \bs{\xi})
	& \defeq
	\begin{bmatrix}
		\vphantom{\Big|}
		L_{2} k(x, \xi_{1}^{\circ}) &
		\cdots &
		L_{2} k(x, \xi_{I}^{\circ}) &
		k(x, \xi_{1}^{\partial}) &
		\cdots &
		k(x, \xi_{J}^{\partial})
	\end{bmatrix} , \\
	\intertext{and}
	T_{1} T_{2} k(\bs{\xi}, \bs{\xi})
	& \defeq
	\hspace{-0.22ex}
	\begin{bmatrix}
		L_{1} L_{2} k(\xi_{1}^{\circ}, \xi_{1}^{\circ}) & \cdots & L_{1} L_{2} k(\xi_{1}^{\circ}, \xi_{I}^{\circ}) & L_{1} k(\xi_{1}^{\circ}, \xi_{1}^{\partial}) & \cdots & L_{1} k(\xi_{1}^{\circ}, \xi_{J}^{\partial}) \\
		\vdots & \ddots & \vdots & \vdots & \ddots & \vdots \\
		L_{1} L_{2} k(\xi_{I}^{\circ}, \xi_{1}^{\circ}) & \cdots & L_{1} L_{2} k(\xi_{I}^{\circ}, \xi_{I}^{\circ}) & L_{1} k(\xi_{I}^{\circ}, \xi_{1}^{\partial}) & \cdots & L_{1} k(\xi_{I}^{\circ}, \xi_{J}^{\partial})\\
		L_{2} k(\xi_{1}^{\partial}, \xi_{1}^{\circ}) & \cdots & L_{2} k(\xi_{1}^{\partial}, \xi_{I}^{\circ}) & k(\xi_{1}^{\partial}, \xi_{1}^{\partial}) & \cdots & k(\xi_{1}^{\partial}, \xi_{J}^{\partial}) \\
		\vdots & \ddots & \vdots & \vdots & \ddots & \vdots \\
		L_{2} k(\xi_{J}^{\partial}, \xi_{1}^{\circ}) & \cdots & L_{2} k(\xi_{J}^{\partial}, \xi_{I}^{\circ}) & k(\xi_{J}^{\partial}, \xi_{1}^{\partial}) & \cdots & k(\xi_{J}^{\partial}, \xi_{J}^{\partial})
	\end{bmatrix} .
\end{align*}
Heuristically, $T$ is the operator $u \mapsto (L u, B u)$.
However, the crucial step missing from the above is to define the operator $T$ properly and to verify that it satisfies the hypotheses of \Cref{thm:main}, so that $T u$ will indeed be a GP to which the conditioning formula can be applied.
The following outlines one such construction, with no attempt at making it the sharpest possible.

Let $\mathcal{H} \subseteq \Reals^{X}$ be a Hilbert space with respect to the usual $L^{2}$ inner product, and assume that point evaluation $\delta_{x}$ at each $x \in X$ is a continuous linear functional on $\mathcal{H}$, i.e.~$\mathcal{H}$ is a reproducing kernel Hilbert space.
Assume also that continuous functions are dense in $\mathcal{H}$.
A good example of this setting is a space of Fourier band-limited functions on $X$.

The operator $u \mapsto (L u, B u)$ is easily defined as an unbounded operator from $\mathcal{H}$ or $L^{2}(X)$ into $L^{2}(X) \oplus L^{2}(\partial X)$ on an $L^{2}$-dense domain such as $\mathcal{C}^{2}(\overline{X}; \Reals)$.
This operator is closable, and even closed in some cases such as $d = 1$ or zero Dirichlet boundary conditions (\citealp[Section~4]{BehrndtLanger2007}; \citealp[Section~III.5.5]{Kato1995}).
Thus, we take $T \colon \dom(T) \subseteq \mathcal{H} \to L^{2}(X) \oplus L^{2}(\partial X)$ to be the closure of this operator, and it possesses a densely defined adjoint as it is defined in a Hilbert space --- although, in fact, one typically shows that $T$ is closable by constructing the densely-defined adjoint explicitly for smooth enough functions.

For simplicity, we assume that the prior covariance kernel $k$ is bounded and $\mathcal{C}^{\infty}$, e.g.~the widely-used squared exponential / radial basis function kernel $k(x, x') \defeq \exp(- \absval{ x - x' }^{2})$ and the rational quadratic kernel $k(x, x') \defeq (1 + \absval{ x - x' }^{2})^{-1}$, since this ensures that the sample paths of $u \sim \GP(0, k)$ are also smooth and a fortiori in $\dom(T)$.
For sharper results connecting the regularity of $k$ to the regularity of draws from $\GP(0, k)$, see \citet[Sections~3 and~4]{Scheuerer2010}.
Also, since this choice of $k$ ensures that $u$ is a.s.~in $\mathcal{C}^{2}(\overline{X}; \Reals)$, Fernique's theorem shows that $\norm{ u }_{\mathcal{C}^{2}}$ is exponentially integrable, yielding finite moments of all orders for $\norm{ u }_{L^{2}(X)}$, $\norm{ L u }_{L^{2}(X)}$, and $\norm{ B u }_{L^{2}(\partial X)}$.

\section*{Acknowledgements}
\addcontentsline{toc}{section}{Acknowledgements}

TM has been supported through the EPSRC Centre for Doctoral Training in Modelling of Heterogeneous Systems (HetSys), grant no.~\href{https://gow.epsrc.ukri.org/NGBOViewGrant.aspx?GrantRef=EP/S022848/1}{EP/S022848/1}.
The authors thank Natha\"el da Costa, Hefin Lambley, Marvin Pf\"ortner, and Daniel Ueltschi for helpful and collegial comments;
we also thank the editors and anonymous peer reviewers for feedback which has improved the manuscript.

For the purpose of open access, the authors have applied a Creative Commons Attribution (CC~BY) licence to any Author Accepted Manuscript version arising.
No data were created or analysed in this study.

\appendix
\section{Proof of \Cref{lem:GP_to_GRV_revised} on Gaussian processes and Gaussian measures}
\label{app:technical}

\Cref{lem:GP_to_GRV_revised} asserts that a Gaussian process on $X$ with its sample paths a.s.~in a ``nice'' Banach space $\U$ induces a Gaussian measure on $\U$.
This is a simple claim, but not a trivial one.
\citet[Proposition~2.3.9]{Bogachev1998} shows that a GP induces a Gaussian measure on the locally convex space $\Reals^{X}$, i.e.~each $\dualprod{ \delta_{x} }{ u } \defeq u(x)$ is Gaussian in $\Reals$.

It has already been noted by \citet[Lemma~A.1]{HaerkoenenEtAl2023} that the continuity of point evaluation makes $\Delta \defeq \SPAN \{ \delta_{x} \}_{x \in X}$ weakly-$\ast$ dense in $\U'$.
\Cref{lem:GP_to_GRV_revised}\ref{item:GP_to_GRV_revised_2} goes beyond this, showing that additionally assuming separability and quasi-reflexivity of $\U$ makes $\Delta$ weakly-$\ast$ \emph{sequentially} dense in $\U'$;
weak-$\ast$ sequential density can also be established ``by hand'' in some other cases of interest such as spaces of compactly-supported continuous functions.
We may then pass from Gaussianity of $\dualprod{ \delta_{x} }{ u }$ to Gaussianity of $\dualprod{ \ell }{ u }$ for $\ell \in \U'$ using a weakly-$\ast$ approximating sequence for $\ell$ and Lebesgue's dominated convergence theorem.
This cannot be done using weak-$\ast$ density alone, since this would yield only an approximating \emph{net} for $\ell$, along which the dominated convergence theorem cannot be applied.

\begin{proof}[Proof of \Cref{lem:GP_to_GRV_revised}.]
	For \ref{item:GP_to_GRV_revised_1}, let $A \subseteq \U'$ be any weakly-$\ast$ sequentially closed subspace containing $\Delta$.
	Since $\U$ is separable, $A$ is also weakly-$\ast$ closed (\citealp[p.124]{Banach1932}; \citealp[Theorem~1.3]{Ostrovskii2001}).
	Thus, if $A \neq \U'$, then there exists $\ell \in \U' \setminus A$ and $u \in \U$ --- which is the dual of $\U'$ with its weak-$\ast$ topology --- such that
	\begin{align}
		\label{eq:HanhBanach1}
		\dualprod{ \ell }{ u } = 1
		\quad
		\text{but}
		\quad
		\dualprod{ a } { u } = 0
		\text{ for all $a \in A$}.
	\end{align}
	However, a special case of \eqref{eq:HanhBanach1} is that
	\begin{align*}
		\dualprod{ \ell }{ u } = 1
		\quad
		\text{but}
		\quad
		\underbrace{ \dualprod{ \delta_{x} }{ u } \equiv u(x) = 0
		\text{ for all $x \in X$} }_{\iff u = 0} ,
	\end{align*}
	which is clearly a contradiction.
	
	For \ref{item:GP_to_GRV_revised_2}, if $\U$ is separable and quasi-reflexive, then the total linear subspace $\Delta \subseteq \U'$ has order at most $1$ \citep[Theorem~2.3(2)]{Ostrovskii2001}.
	That is, $(\Delta_{(1)})_{(1)} = \Delta_{(1)}$, i.e.\ $\Delta_{(1)}$ (or possibly even $\Delta$ itself) is weakly-$\ast$ sequentially closed, and hence equals $\U'$ by \ref{item:GP_to_GRV_revised_1}.
	
	In the case that $\U$ is a separable Banach space of bounded functions with respect to the supremum norm, and that element of $\U$ attains its supremum norm, we appeal to \citet[p.213]{Banach1932} (see also \citet[Theorem~1.4]{Ostrovskii2001}):  $\Delta_{(1)} = \U'$ since there exists $M > 0$ such that, for all $u \in \U$, there exists $\ell \in \Delta$ with $\norm{ \ell } \leq M$ and $\absval{ \dualprod{ \ell }{ u } } = \norm{ u }$;  one can simply take $M = 1$ and $\ell = \delta_{x_{0}(u)}$, where $x_{0}(u) \in X$ is any point at which $u$ attains its supremum norm.
	
	For \ref{item:GP_to_GRV_revised_3}, by \citet[Theorem~2.2.4]{Bogachev1998}, it suffices to show that there exist a (not necessarily continuous) linear function $L$ and a non-negative quadratic form $Q$ on $\U'$ such that
	\begin{align*}
		\log \E [ \exp( \iunit \dualprod{ \ell }{ u } ) ] = \iunit L (\ell) - \frac{1}{2} Q(\ell)
		\quad
		\text{for all $\ell \in \U'$.}
	\end{align*}
	To that end, let $\ell \in \U'$ be arbitrary.
	Since $\U' = \Delta_{(1)}$, there exists a sequence $(\ell_{n})_{n \in \Naturals}$ in $\Delta$ converging weakly-$\ast$ to $\ell$;
	write $\ell_{n} = \sum_{j = 1}^{J(n)} \alpha_{n, j} \delta_{x_{n, j}}$ for some $\{ \alpha_{n, j} \}_{j = 1}^{J(n)} \subset \Reals$ and $\{ x_{n, j} \}_{j = 1}^{J(n)} \subseteq X$.
	Then, appealing to Lebesgue's dominated convergence theorem for the interchange of the limit and the expectation,
	\begin{align*}
		\log \E [ \exp( \iunit \dualprod{ \ell }{ u } ) ]
		& = \log \E \left[ \exp \left( \iunit \Dualprod{ \mathop{\textup{wk-$\ast$-lim}}_{n \to \infty} \ell_{n} }{ u } \right) \right] \\
		& = \lim_{n \to \infty} \log \E [ \exp( \iunit \dualprod{ \ell_{n} }{ u } ) ] \\
		& = \lim_{n \to \infty} \left( \iunit \sum_{j = 1}^{J(n)} \alpha_{n, j} m_{u}(x_{n, j}) - \frac{1}{2} \sum_{j, j' = 1}^{J(n)} \alpha_{n, j} \alpha_{n, j'} k_{u}(x_{n, j}, x_{n, j'}) \right)
	\end{align*}
	It is neither difficult nor fun to verify that the right-hand side has the same limiting value regardless of the approximating sequence $(\ell_{n})_{n \in \Naturals}$ that is used and that approximating sequences may be chosen as linear functions of the approximands $\ell$ (perhaps after taking common refinements of the nodal sets $\{ x_{n, j} \}_{j}$).
	Thus, as a convergent limit of quadratic forms, the left-hand side is also a quadratic form in $\ell$, which shows that $u$ is weakly Gaussian.
	
	Finally, if $u$ is essentially separably valued, then the claim about strong Gaussianity follows from Pettis' measurability theorem.
\end{proof}

\bibliographystyle{abbrvnat}
\bibliography{references}
\addcontentsline{toc}{section}{References}

\end{document}

%% file: doc-packages.tex
\usepackage{amsfonts,amssymb,amsmath,amsthm,amstext,amssymb,mathtools}
\usepackage{subcaption,float,color,xcolor,graphicx,enumitem}
\usepackage{dsfont}  
\usepackage[normalem]{ulem} 
\usepackage[authoryear,sort,round]{natbib}  
\usepackage{hyperref,nicefrac}
\usepackage[noabbrev,capitalise,nosort,nameinlink]{cleveref}

\crefname{assumption}{Assumption}{Assumptions}

\addtokomafont{disposition}{\sffamily}

\usepackage{csquotes}

\usepackage{graphics}
\usepackage{tikz, tikz-cd}
\usetikzlibrary{decorations.pathreplacing,calc}

%% file: doc-newcommands.tex
\DeclareMathOperator*{\dom}{\mathcal{D}}

\DeclareMathOperator*{\graph}{\mathcal{G}}
\DeclareMathOperator*{\partitions}{part}

\newcommand*{\bs}{\boldsymbol}

\newcommand*{\Cov}{\mathbb{C}\mathrm{ov}}
\newcommand*{\defeq}{\coloneqq}
\newcommand*{\defterm}{\textbf}
\newcommand*{\E}{\mathbb{E}}

\renewcommand*{\epsilon}{\varepsilon}
\newcommand*{\GP}{\mathcal{GP}}

\newcommand*{\hatotimes}{\mathbin{\hat{\mathord{\otimes}}}}
\newcommand*{\iunit}{\mathrm{i}}
\newcommand*{\Naturals}{\mathbb{N}}
\newcommand*{\Normal}{\mathcal{N}}
\newcommand*{\one}{\mathds{1}}
\renewcommand*{\P}{\mathbb{P}}

\newcommand*{\quark}{\setbox0\hbox{$x$}\hbox to\wd0{\hss$\cdot$\hss}}
\newcommand*{\rd}{\mathrm{d}}

\newcommand*{\Reals}{\mathbb{R}}
\DeclareMathOperator*{\SPAN}{span}

\newcommand*{\U}{\mathcal{U}}
\newcommand*{\V}{\mathcal{V}}

\renewcommand*{\geq}{\geqslant}

\renewcommand*{\leq}{\leqslant}

\renewcommand*{\mathbf}{\boldsymbol}

\newcommand{\todo}[1]{\bgroup\color{red}#1\egroup}
\newcommand{\TJSsays}[1]{\bgroup\color{olive}Tim:~#1\egroup}
\newcommand{\TMsays}[1]{\bgroup\color{green}Tadashi:~#1\egroup}

\theoremstyle{plain}
\newtheorem{theorem}{\sffamily Theorem}[section]

\newtheorem{lemma}[theorem]{\sffamily Lemma}

\theoremstyle{definition}

\newtheorem{example}[theorem]{\sffamily Example}

\newtheorem{remark}[theorem]{\sffamily Remark}

\newcommand{\absval}[1]{\lvert #1 \rvert}
\newcommand{\dualprod}[2]{\langle #1 \vert #2 \rangle}

\newcommand{\norm}[1]{\lVert #1 \rVert}
\newcommand{\set}[2]{\{ #1 \mid #2 \}}

\newcommand{\Dualprod}[2]{\left\langle #1 \middle\vert #2 \right\rangle}

\numberwithin{equation}{section}
\numberwithin{figure}{section}
\numberwithin{table}{section}

%% file: chunk-abstract.tex
Gaussian processes (GPs) are widely-used tools in spatial statistics and machine learning and the formulae for the mean function and covariance kernel of a GP $T u$ that is the image of another GP $u$ under a linear transformation $T$ acting on the sample paths of $u$ are well known, almost to the point of being folklore.
However, these formulae are often used without rigorous attention to technical details, particularly when $T$ is an unbounded operator such as a differential operator, which is common in many modern applications.
This note provides a self-contained proof of the claimed formulae for the case of a closed, densely-defined operator $T$ acting on the sample paths of a square-integrable (not necessarily Gaussian) stochastic process.
Our proof technique relies upon Hille's theorem for the Bochner integral of a Banach-valued random variable.

%% file: chunk-keywords.tex
Bochner integral%
\and%
closed operator%
\and%
densely-defined operator%
\and%
Gaussian process%
\and%
Hille's theorem%
\and%
unbounded operator%

%% file: chunk-msc.tex
60G12
\and%
60G15
\and%
46G10
\and%
47B01